\documentclass[11pt,dvips]{article}
\usepackage{amssymb,fullpage}
\usepackage{xspace}
\usepackage{latexsym}
\usepackage{times}
\usepackage{amsfonts}
\usepackage{amsmath}
\usepackage{amsthm}

\def\Q{\mathbb{Q}}

\def\F{\mathbb{F}}
\def\T{{\cal{T}}}

\def\cl{{\rm {span}}}
\def\cli{{\rm {cl}^*}}

\newtheorem{theorem}{\bf Theorem}
\newtheorem{claim}[theorem]{\bf Claim}

\newtheorem{corollary}[theorem]{\bf Corollary}
\newtheorem{definition}{\bf Definition}

\newcommand{\supp}{{\rm supp}}
\newcommand{\rank}{{\rm rank}}

\newcommand{\ignore} [1] {} %

\title{Large Simple $d$-Cycles in Simplicial Complexes}
\author{Roy Meshulam \thanks{Department of Mathematics,
Technion, Haifa 32000, Israel. e-mail:
meshulam@technion.ac.il~. Supported by ISF grant 326/16.}
\and Ilan Newman \thanks{Department of Computer Science, University
of
Haifa, Haifa, Israel. E-mail: {\tt ilan@cs.haifa.ac.il}. 
This Research
was supported by The Israel Science Foundation, grant number
497/17} \and Yuri Rabinovich\thanks{Department of Computer Science, University
of
Haifa, Haifa, Israel. E-mail: {\tt yuri@cs.haifa.ac.il}
}
}

\begin{document}
\maketitle
\begin{abstract}
  We show that the size of the largest simple $d$-cycle $C$ in a
  simplicial $d$-complex $K$ is at least a square root of $K$'s
  density. This generalizes a well-known classical result of Erd\H{o}s
  and Gallai~\cite{EG59} for graphs. We use methods from matroid
  theory, applied to combinatorial simplicial complexes.
\end{abstract}

\section{Introduction}
\label{s:intro}

Let $G=(V,E)$ be a finite graph. A classical result of Erd\H{o}s and Gallai~\cite{EG59} asserts that
if $|E| > 2k(|V|-1)$, then $G$ contains a simple cycle of length \;$> k$.
In this paper we study the analogous question for higher dimensional simplicial complexes.

A $d$-cycle in a $d$-dimensional simplicial complex is a set of $d$-faces with coefficients in a ring $R$,
whose {\em boundary} is $0$.\footnote{See Section~\ref{sec:defs} for a precise definition.}
A $d$-cycle is {\em simple} if it contains (in a set theoretic sense, i.e., ignoring the coefficients) no other nontrivial
$d$-cycles. While in graphs the choice of the ring $R$ is immaterial as far as the structure of the simple cycles is concerned, in higher dimensions it is of importance. In this paper we assume that $R$ is an arbitrary field $\F$. One advantage of working over a field is that it introduces a matroidal structure on the set of $d$-faces of a complex, where a subset of
$d$-faces is {\em independent} if it supports no nontrivial $d$-cycles.
Our results will exploit the combinatorial structure of $d$-cycles, and will not dependent on the choice of $\F$.

Let $c(G)$ denote the length of the maximum simple cycle in $G$.
The graph-theoretic lower bound of Erd\H{o}s and Gallai~\cite{EG59} can be interpreted it in two somewhat
different ways. The first interpretation is that $c(G)$ is linear in $D=2|E|/|V|$, the average degree of $G$.
The second interpretation is that $c(G)$ is linear in $|E|/\rank(G)$,
where the rank of $G$  is the
size of a maximum acyclic subset of edges in it. The latter interpretation is more suitable for a generalization,
and we shall pursue it for the most part of the paper; it will also be used to obtain a generalization of the former interpretation.

Observe that the results in~\cite{EG59} further implies that $c(G)$ is linear not only in $|E|/\rank(G)$, but also in
$\gamma(G) = \max_{G' \subseteq G} \lceil |E(G')|/\rank(G') \rceil$, where the maximum is taken over all
subgraphs of $G$.
This is a standard graph theoretic paramater. E.g., by Nash-Williams Theorem~(see, e.g., \cite{Diestel}), the minimum number
of subforests of $G$ required to cover $E(G)$  is precisely $\gamma(G)$. Moreover, it generalizes to matroids, and hence to simplicial complexes over a field. 
A classical result of Rado (see e.g., \cite{Oxley11}) asserts that the minimum number of independent sets in a
matroid $M$ required to cover $M$ is ~$\gamma(M)= \max_{A \subset E(M) }\lceil {|A|}/{\rank(A)} \rceil.$

In this paper we develop a general matroid theoretical framework allowing to obtain lower bounds on $c(M)$,
the size of the largest simple cycle of $M$, in terms of $\gamma(M)$. We then adapt and augment this
technique to obtain similar bounds for simplicial complexes. One representative result is:
\vspace{0.2cm}

{\bf Theorem} {\em Let $K$ be a pure simplicial $d$-complex containing
  $d$-cycles. Let $f_\ell(K)$ be the number of $\ell$-simplices in $K$. Then, $K$ contains
  a simple $d$-cycle of size
  $~\geq\; 1/ \sqrt{d}\cdot \sqrt{ { {f_d(K)} / {f_{d-1}(K)}}} -1 $.}
\vspace{0.2cm}

This paper contributes to the rapidly evolving study of the combinatorics of simplicial complexes in the context of their homological and homotopical properties. Let us mention, e.g., the paper~\cite{DGK} dealing with {\em small}
(simple) cycle in dense simplicial complexes, addressing a similar (actually, only a similar-looking) problem.
\section{Standard Notions Pertaining to Simplicial Complexes and to Matroids}
\label{sec:defs}
\noindent
{\bf Simplicial complexes:}~
Let $K$ be a finite $d$-dimensional simplicial complex on the vertex set $V$ and let $ \F$ be a field. Let $\prec$ be a fixed linear order on $V$. Orient each simplex in $K$ according to this order, i.e.,
$\sigma=[v_0,\ldots,v_i]$ if $v_0 \prec \cdots \prec v_i$.
Let $K^{(i)}$ denote the set of oriented $i$-dimensional faces of $K$, and let
$f_i(K)=|K^{(i})|$. Let $C_i(K; \F)$ be the space of $i$-chains of $K$, where a chain is a free
sum (i.e., a union) of weighted oriented $i$-simplices in $K^{(i)}$.

The {\em boundary} of $d$-simplex $\sigma = [v_0,\ldots,v_d]$ is ~ $\partial_{d}(\sigma) = \sum_{j=0}^d (-1)^j [v_0,\cdots,\tilde{v}_j,\cdots,v_d]$ where $\tilde{v}_j$ stands for an
omitted $v_j$. This linearly extends to the {\em boundary map}
$\partial_d:C_d(K; \F) \rightarrow  C_{d-1}(K; \F)$.

As usual, $Z_d(K; \F)=\ker \partial _d$ is the linear space over $\F$ of $d$-cycles in $K$,
and $B_{d-1}(K; \F)=\text{Im}\,\partial_d$ is the linear space over $\F$ of $(d-1)$-boundaries in $K$.

Let $C=\sum_{i=1}^m \alpha_i \sigma_i \in Z_d(X; \F)$ be a $d$-cycle in $K$, where the
$\alpha_i \neq 0$, and the $\sigma_i$'s are distinct $d$-simplices.
A $d$-cycle $C$ is \emph{simple} if the set
$\{\partial_d\sigma_1,\ldots,\partial_d\sigma_{i-1},\partial_d\sigma_{i+1},\ldots,\partial_d\sigma_m\}$ is linearly independent over $ \F$ for each $1 \leq i \leq m$.
Equivalently, $C$ is simple if its {\em support} does strictly contain the support of any other
non-trivial $d$-cycle in $K$. The {\it support} of $C$ is
$\supp(C)=\{\sigma_1,\ldots,\sigma_m\}$.
\\ \\
{\bf Matroids:} We only list here some of the more relevant notions from the matroid theory.
For more details see Oxley's book \cite{Oxley11}.  Let $M  = (E,
\mathcal{I})$ be a matroid on $E$, with $\mathcal{I}$ its collection
of  independent sets .
The {\em rank} of $A \subseteq E$ is the size of the maximum independent set in it.
The span (also called the closure) of $A$ in $M$, denoted by
$A \subseteq \cl(A) \subseteq E$, is the maximal set containing $A$ such that $\rank(A) = \rank(\cl(A))$.
A {\em circuit} is a minimally dependent set.

A matroid $M$ is called {\em loopless} if no $e \in E(M)$ forms a circuit. It has no {\em double edges} if no
$e,f \in E(M)$ form a circuit. A loopless matroid without double edges is called {\em simple}.

 A matroid $M$ is called {\em connected} if it is not a direct sum of its
{\em submatroids}, where a submatroid is a matroid on a subset of $E(M)$ with the rank function inherited from
$M$. Equivalently, define a binary relation on $E(M)$ where $e \sim f$ if there exists a circuit containing both.
This turns out to be an equivalence relation. The equivalence classes of this relation are called the {\em connected
components} of $M$. The matroid $M$ if connected if there is only one such component.

{\em Minors:} A minor of a matroid $M$ is a matroid obtained from $M$
by series of {\em deletions} and {\em contractions}. Deletion of $e \in E(M)$ is the matroid $M \setminus e = (E\setminus \{e\},
\mathcal{I'})$ where $\mathcal{I'} = \{I \subseteq E \setminus \{e\}|~
I  \in \mathcal{I} \}$.  Contraction of $A$ is the matroid $M / e = (E
\setminus \{e\}, \mathcal{I'})$ where $\mathcal{I'} = \{ I \subseteq E
\setminus \{e\}~ | ~ I \cup \{e\} \in \mathcal{I} \}$.

The elements of a {\em linear} matroid over $\F$ are vectors over $\F$, and the rank function is the usual linear-algebraic rank. A matroid is called $\F$-representable if it is isomorphic
to a linear matroid over $\F$. The elements of a {\em simplectic} matroids are
the $d$-simplices of a simplicial complex  $K$, where a set of such simplices is dependent
whenever it contains a support of a nontrivial $d$-cycle of over $\F$. Simplectic matroids are linear.

\section{Approaches and Results: Matroids}
\label{sec:main}
\subsection{Using Forbidden Minors: $\F_q$-representable Matroids}
Observe that a class of graphs $G$ with $c(G) < k$ is closed under minors. This fact could be employed to
obtain a weaker version of~\cite{EG59} by using classical results (see, e.g.,~\cite{Thomason}
about the density of graphs lacking a size-$k$ minor).

As with graphs, the class of matroids $M$ with $c(M) < k$ lacks the graphic minor $M(C_k)$, i.e., the matroid
associated with the graph $C_k$, and hence it lack also $M(K_k)$, i.e., the matroid associated with the graph $K_k$.
A hard result of Geelen and Whittle~\cite{GW03} (see also~\cite{Geelen11}) asserts that
\begin{theorem}{\bf\cite{GW03}} Let $M$ be a simple $\F_q$-representable matroid (equivalently, a linear matroid over a field of size $q$) lacking $M(K_k)$. Then,
\[
\gamma(M) ~<~ q^{q^{3k}}~.
\]
\end{theorem}
\begin{corollary}
\label{cor:q-rep}
For $M$ as above, ~$c(M) \; >~ {1\over 3}\log_q\log_q \gamma(M)\,$.
\end{corollary}
While this does establish a weak lower bound on $c(M)$ in terms of $\gamma(M)$ for, say, binary matroids,
the bound can be considerably strengthened (see below). For infinite fields like $\Q$ it yields nothing.
\subsection{Using Seymour's Lemma: General Matroids}
The following lemma by Seymour (Theorem 3.4 in~\cite{DOO95}) will be used.
\begin{theorem}
\label{th:seymour}
Let $M$ be a connected loopless matroid, $|E(M)|>1$,  and let $C$ be a maximum size circuit in $M$.
Then, size of the maximum circuit in the induced matroid  $M/C$ is strictly less than $|C|$.
\end{theorem}
For matroids with loops, $\gamma()$ is not interesting as formally
$\gamma(M)=\infty$. For loopless matroids, as far as the relations between $c(M)$ and $\gamma(M)$ goes, we may  restrict our
attention to connected matroids. Indeed, assume that $M$ has connected
components $\{M_i\}_1^{\ell}$. Then, on one hand, a circuit of $M_i$ is also a circuit of $M$, and so
$c(M) \geq c(M_i)$. In fact, $c(M) =\max_i c(M_i)$, since any circuit of $M$ lies in some connected component.  On the other hand,  it holds that $\gamma(M) = \max_i \gamma(M_i)$.
The direction "$\geq$" is obviously true. For the direction "$\leq$",
let $K \subseteq M$ be the subset of elements on which $\gamma(M)$ is achieved, and let
$K_i = K \cap E(M_i)$. Then,
\[
\gamma(M) ~=~ {{|K|} \over {\rank(K)}} ~=~ {{\sum_j |K_j|}\over{\sum_j \rank(K_j)}} ~\leq~
\max_j {{|K_j|}\over{ \rank(K_j)}} ~\leq~\gamma(M_j)~\leq~\max_i\,\gamma(M_i)~.
\]
Given a connected loopless matroid $M$, $|E(M)|>1$, we define the following decomposition process of $M$,
described by a tree $\T_M\,$:
\begin{definition}
\label{def:T}
Each vertex $x$ of $\T_M$ will have an associated pair $(M_x,C_x)$,
where $M_x$ is a connected loopless  minor of $M$, and $C_x$ is
a maximum size circuit of $M_x$.  

The matroid associated with the root vertex is the original $M$.

The children of the vertex $x$ in $\T_M$ correspond to the connected components of $M_x/C_x$ after removal
of the loops. If there are no such non-empty components (that is, $rank(M_x
/ C_x)=0$), $x$ is a leaf of  $\T_M$.
\end{definition}
The following claim establishes some basic properties of $\T_M$.
\begin{claim}
\label{cl:T_M}
$\mbox{}$

* ~ $\T_M$ is well defined, given an arbitrary maximum size circuit
$C_x$ at each vertex.

\vspace{0.2cm}
* ~ For any $x,y$, where $y$ is the parent of $x$ in $\T_M$, it holds that $|C_y| > |C_x|$. \\
$\mbox{}$  \hspace{0.91cm} Consequently, the depth of $\T_M$\; is $\;<~ c(M) -1$\,.

\vspace{0.2cm}
* ~ $\sum_{x\in \T_M} (|C_x| - 1) ~=~ \rank(M)$.
\end{claim}
\begin{proof}
Let us first observe that $rank(M_x) \geq 1$ at each $x \in \T_m$. This is
true for the root, by assumption, and since we remove loops after
every contraction step, either the result is an empty matroid, or it
is of rank at least $1$.  Next, let us verify that $\T_M$ is well
defined - namely, that for any vertex generated in the process
defining $\T_M$, we can proceed by  reducing $M_x$, unless $x$ is a
leaf. The only problem that
may occur is when  $M_x$ does not contain a circuit.
Assume by contradiction that such $M_x$'s exist. Let $x$ be the highest such vertex in $\T_M$.
Recall that in a connected matroid every element $e$ is contained in a circuit, unless $e$ is a sole element there.
Since $M_x$ is connected, loopless and nonempty, it must be of the form $E(M_x)=\{e\}$, where
$e$ is not a loop. By our assumptions $|E(M)|>1$, and therefore $x$ is not a root; let $y$ be its parent in $\T_M$.

By the choice of $x$, $|E(M_y)| > 1$, and since $M_y$ is loopless and connected, there is a circuit $C$ in $M_y$
containing $e$ and some other elements. Contraction $C_y$ in $M_y$,
the set $C \setminus C_y$ splits into disjoint circuits. Since $e$ is
not a loop in $M_x$, $e$ is contained in some circuit $C' \subseteq C$
in $M_y/C_y$ with $|C'| > 1$. Hence its connected component contains
at least one additional element besides itself contradicting the fact
that $|E(M_x)|=1$.

The second statement is an immediate consequence of Theorem \ref{th:seymour}.

The third statement can be shown by bottom-up induction on the
structure of $\T_m$. Observe that the subtree of $\T_M$ rooted at the
vertex $x \in V(\T_M)$ is
$\T_{M_x}$.
  Hence it is enough to show this for the leaves of $\T_m$, and then for any node $x$ in $\T_M$ assuming the
statement holds for $x$'s descendants.

 When for a leaf $x$, by definition $|C_x|=1$ while $rank(M_x)=0 =
 |C_x|-1$.  Consider now a vertex
$y \in \T_M$ with  children $\{x_i\}_1^\ell$. Since loops removal has
no effect on the rank, using the fact that a non-simple matroid has
rank that is the sum of ranks of its components, we get
\[
\rank(M_y) ~=~ \rank(C_y) \;+\; \rank(M_y/C_y) ~=~ (|C_y| - 1) \;+\;  \sum_i \rank(M_{x_i})
\]
\vspace{-0.5cm} \\ $\mbox{}$
\hspace{1.8cm}
$~=~ (|C_y| - 1) \;+\;
\sum_{z\in \T_{M_{x_i}}} (|C_z| - 1)
=~ \sum_{z\in \T_{M_y}} (|C_z| - 1)\,.
$
\end{proof}
In order to relate $c(M)$ to $\gamma(M)$, we introduce the following matroid-theoretic function $s_M(i)$ of $M$:
\begin{definition}
For $i \geq 0$, let ~~$s_M(i)=\max\{\, |K|: K \subset E(M), \;\rank(K) \leq i \}$\,.
\end{definition}
The following theorem is one of the central results of this paper:
\begin{theorem}
\label{t:covmat}
Let $M$ be a loopless matroid with $c(M)=k>1$. Then,
\[
\gamma(M) ~\leq~ s_M\left({{(k-1)k} \over 2} \right)\,.
\]
\end{theorem}
\begin{proof}
Let $N$ be the submatroid of $M$  on which $\gamma(M)$ is achieved.
That is, $|E(N)|=\gamma(M) \cdot\rank(N)$. As we have seen, 
$N$ is w.l.o.g., connected, loopless, with $|E(N)| > 1$.  We shall prove that
\[
|E(N)| ~\leq~ s_N \left({{(r-1)r} \over 2}\right) \cdot \rank(N)\,,
\]
where $r=c(N) \leq c(M) = k$. Observe that $s_N(*)$ is dominated by $s_M(*)$.

Let $\T_N$ be the decomposition tree of $N$, and consider a vertex $x$
of $\T_N$, and its father $y$ in $\T_N$.
Observe that $e \in E(N)$ becomes a loop when defining $N_x$, if and
only if  it is spanned
by \;$\cup_{z \in P_y} C_z$\;, where $P_y$ is the path from $y$ to the root in $\T_N$.
Indeed, one can verify this inductively, recalling that
$e \in E(N)$ becomes a loop in $N/A$,  if and only
if $e\in \cl(A)$, and using the  identity
$(N/A)/B \cong N/(A\cup B)$ along with the fact that $N_x$ is isomorphic to a
connected component of $N / \cup_{z \in P_y} C_y$.

Let $L(\T_N)$ be the set of leaves of $\T_n$, and for $z \in L(\T_n)$
let $P_z$ be the path from $z$ to the root of $\T_N$.
Keeping in mind that all the elements of $N$ get eventually eliminated during the
decomposition process described by $\T_N$, one concludes that
\begin{equation}
\label{eq:union}
\bigcup_{z:\, {\rm leaves~of}~\T_N} \cl\left( \bigcup_{x \in P_z} C_x \right) ~~=~~ E(N)\,,
\end{equation}

Now, for any $z \in L(\T_N)$,
\begin{equation}\label{eq:22}
\rank \left( \cl\left( \bigcup_{x \in P_z} C_x \right)\right) ~~=~~
\rank \left( \bigcup_{x \in P_z} C_x \right)
~=~ \sum_{x \in P_z} |C_x| - 1\,.
\end{equation}

Since $c(N) \leq c(M)= k$, then $|C_x| \leq k$ for every vertex
$x \in \T_n$. In addition, Theorem \ref{th:seymour} implies that the
size of $C_x$  drops down by at least $1$ every time we move
down the tree. Therefore, for any leaf $z \in L(\T_n)$, $\sum_{x \in
  P_z} |C_x| - 1 \leq (r-1) + (r-2) + \ldots + 1 = (r-1)r/2$.\;
Recalling the definition of  $s_M(*)$ this implies that,
\begin{equation}
\label{eq:2}
\Bigg| \cl\left( \bigcup_{x \in P_z} C_x \right) \Bigg|  ~~\leq~~  s_N \left({{(r-1)r} \over 2}\right)     \,.
\end{equation}

In view of Claim~\ref{cl:T_M}, the number of vertices of $\T_N$, and
in particular the number of leaves there, is at most
$\rank(N)$. Combining this  with Equations (\ref{eq:union}) and
(\ref{eq:2}) we conclude that
\[
|E(N)| ~~\leq
\sum_{z:\, {\rm leaves~of}~\T_N} \Bigg|\cl\left( \bigcup_{x \in P_z} C_x
\right)\Bigg| ~~\leq~~ \rank(N) \cdot s_N \left({{(r-1)r} \over 2}\right)\,,
\]
as desired.
\end{proof}
For an application of Theorem \ref{t:covmat},  consider the case when $M$ is a $\F_q$-representable matroid, i.e.,  is isomorphic to a linear matroid over $\F_q$. In this case $s_M(r) \leq q^r$. Then, applying
Theorem~\ref{t:covmat}, we conclude that
\begin{corollary}
\label{cor:q-rep1}
For $M$ as above, $\gamma(M) ~\leq~ q^{{c(M) \choose 2}}$. ~
Consequently,  ~$c(M) \; >~ \sqrt{2  \log_q \gamma(M)}\,$.
\end{corollary}
This is a considerable improvement over Cor.~\ref{cor:q-rep}. For a simple exponential
lower bound, let $M$ be a $k$-dimensional linear space over $\F_q$ with $0$ removed.
For this $M$,~ $\gamma(M) = (q^k - 1)/k$, while $c(M) = k+1$.
\section{Results: Simplicial Complexes}
The general results obtained in the previous section apply to simplicial complexes.
Let us first relate the matroidal notation used in the previous section to the usual notation used for
simplicial complexes.

Let $K$ be a pure $d$-dimensional simplicial complex.
The $f$-vector of $K$ is $(f_d,f_{d-1},\ldots,f_0)$ where $f_i =  |K^{(i)}|$.
The rank function $\rank_d$\; that introduces a matroidal structure on $K^{(d)}$ is
$$\rank_d(A) ~=~ \rank \left\{ \partial \sigma \;|\;  \sigma \in A \subseteq K^{(d)} \right\}.$$
This results in an $\F$-representable\footnote{To be more explicit,
each $\sigma \in K^{(d)}$ is associated with the (signed) incidence vector of
$\partial \sigma$ with respect to $K^{(d-1)}$, and the $\rank$ is the usual linear-algebraic
rank function of sets vectors in $\F^{f_{d-1}}$.}   matroid $M(K^{(d)})$.
In particular, $\mathcal{I}$, the independent sets in $M(K^{(d)})$, are precisely
the acyclic sets of $d$-simplices in $K$. Observe also that  $\rank_d(M(K^{(d)})) =
\dim B_{d-1}$, where $B_{d-1}$ is the space of the $(d-1)$-boundaries of $K$.

We shall use $c_d(K)$ and $\gamma_d(K)$  to denote the size of the largest circuit, and the
value of the parameter $\gamma$ in the above matroid $M(K^{(d)})$,
respectively. Note that $c_d(K)$ coincides with the size of the largest simple $d$-cycle in
$K$, as defined in Section \ref{sec:defs}. Slightly abusing the notation, we shall use $\rank_d(K)$ to denote
$\rank_d(M(K^{(d)}))$.
\\ $\mbox{}$

In order to employ Theorem~\ref{t:covmat}, we shall estimate $s_d(t)$, the maximum possible size of a family $A$ of $d$-complexes with $\rank_d(A) \leq t$. The results of~\cite{BK} implicitly imply the following:
\begin{theorem} {\bf [{implicit in}~\cite{BK}\footnote{This 
was also independently shown by Nati Linial and Yuval Peled, private communication, by a direct argument involving the shifting technique.
A similar, marginally weaker (yet sufficient for our needs here)  but  rather straight-forward  conclusion follows directly from the Kruskal-Katona Theorem, which provides a tight lower bound on $f_{d-1}$ in terms of $f_d$. One only needs to add an observation that ~~$\rank_d(K) \geq {1 \over {d+1}} f_{d-1}(K)$.}
\hspace{-0.14cm} ]}%
\label{th:KK}
$\mbox{}$~ Among all families $B$, $|B|=s$, of $d$-complexes, the
minimum rank is achieved by the family that is
{\em compressed } with respect to the co-lex order (that is, as in the Kruskal-Katona Theorem) $B_0$. The basis is formed by the $d$-simplices in $B_0$ containing the vertex with the smallest label. Using the relaxed
formulation of Kruskal-Katona Theorem by Lovasz~\cite{Lovasz07}, for such $B_0$ it holds that
\begin{equation}
\label{eq:LNPR}
 s ~=~ {{x+1} \choose {d+1}} ~~~~~~~~~~~~\implies
~~~~~~~~~~~~ \rank_d(B_0) ~\geq~ {{x} \choose  d}
\end{equation}

\end{theorem}
Consequently, there exists a constant $a_d$  such that for every
$d$-simplicial complex,
\begin{equation}
\label{eq:s(t)}
s_d(t) ~~\geq~~ a_d\cdot t^{{d+1} \over d}\,.
\end{equation}
Combining Theorem~\ref{th:KK} with Theorem~\ref{t:covmat}, we get
\begin{theorem}
\label{t:covcomp}
There is a universal constant $b_d$, such that for any $d$-dimensional
simplicial complex $K$ containing nontrivial $d$-cycles, it holds that
\[
f_d(K) ~>~ b_d \cdot k^{2+{2\over d}} \cdot \rank_d(K)~~~ \implies ~~~\mbox{ $K$ contains a
simple $d$-cycle of size $> k$}\,.
\]
\end{theorem}
This theorem can be strengthened.
%
\begin{theorem}
\label{t:d=2}
Let $K$ be a simplicial $d$-complex containing nontrivial $d$-cycles. Then,
\[
f_d(K) ~>~ {{d+1} \over 2}\cdot k(k+1) \cdot \rank_d(K)~~~ \implies ~~~\mbox{ $K$ contains a
simple $d$-cycle of size $> k$}.
\]
Consequently, $c_d(K) \geq \sqrt{\gamma_d(K) /d} \;-1$.
\end{theorem}
\begin{proof}
We shall use the tree $\T_K = \T_{M(K)_d}$ as in
Definition~\ref{def:T}, with the notation introduced
in the proof of Theorem~\ref{t:covmat}. It will also be assumed w.l.o.g., that $K$ is pure,
and that the corresponding matroid is connected.

We use the following notations.
The {\em upper shadow} $\Delta^{d-1} (X) \subseteq K^{(d)}$ of $X \subseteq K^{(d-1)}$
is the set of all $d$-simplices $\sigma \in K$ so that $\sigma^{(d-1)} \subseteq X$.
The {\em lower shadow} $\Delta_d (A) \subseteq K^{(d-1)}$ of $A \subseteq K^{(d)}$
is $\Delta_d (A) = \{\sigma^{(d-1)} ~ |~ \sigma \in A \}$.
Finally, for $A\subseteq K^{(d)}$ as above, define
\[
\cli(A) ~\stackrel{def}{=}~ \Delta^{d-1}(\Delta_d (A)) ~=~ \{\,\mbox{The $d$-simplices $\sigma \in K^{(d)}$
with $\sigma^{(d-1)} \in \Delta_d(A)$}~\}
\]
Clearly,~ $\cl(A) \subseteq \cli(A)$.

Let $P_v$ be, as before, the path from the vertex $v$ to the root of $\T_K$.
Define
\[
S_v ~=~ \bigcup_{x \in P_v,\, x\neq v} C_x  ~~\subseteq~ K^{(d)}\,.
\]
In particular, for the root vertex $r$, $S_r = \emptyset$.

Observe that
\begin{equation}
\label{eq:size1}
\bigcup_{x\in \T_M} \cl(S_x \cup C_x) \setminus \cl(S_x) ~~=~ K^{(d)}\,;
\end{equation}
\begin{equation}
\label{eq:size2}
\sum_{x\in \T_M} |\cli(S_x \cup C_x)| - |\cli(S_x)| ~~\geq ~|K^{(d)}|~~=~~ f_d(K)\,.
\end{equation}
Indeed, every $\sigma \in K^{(d)}$ either belong to some $C_x$, or it
is removed as a self loop after contracting $C_x$, for $x \in \T_K$.
Denote by $x(\sigma)$ the vertex $x$ for which the above happen for
$\sigma$ (note that $x(\sigma)$ is unique and well defined as $\sigma$
is removed after processing $x$ in $\T_K$). Since for every
$\sigma \in K^{(d)}$,
$\sigma \in \cl(S_x \cup C_x) \setminus \cl(S_x)$ for $x = x(\sigma)$,
(\ref{eq:size1}) follows.

Fix $\sigma \in K^{(d)}$ and $x =
x(\sigma)$. Since $\sigma \in \cl(S_x \cup
C_x)$ it follows that $\sigma \in \cli(S_x \cup
C_x)$. Hence $\sigma^{(d-1)} \subseteq  (S_x \cup
C_x)^{(d-1)}$.  Let $y \in P_{x(\sigma)}$ be the top-most vertex for
which $\sigma^{(d-1)} \subseteq  (S_y \cup
C_y)^{(d-1)}$. Note that for this $y$, $\sigma$ contributes $1$
to $|\cli(S_y \cup C_y)|$ and $0$ to $|\cli(S_y)|$ hence (\ref{eq:size2}) follows.

Next, we upper-bound the term $|\cli(S_x \cup C_x)| - |\cli(C_x)|$ as follows.
Note that if $\sigma \in \cli(S_x
\cup C_x) \setminus \cli(S_x)$ then there must be $\tau \in
\sigma^{(d-1)} \setminus S_x^{(d-1)}$. This implies that such $\tau$
must be in $C_x^{(d-1)}$.  Since   all the vertices of $\sigma$
belong to $(S_x \cup C_x)^{(0)}$, it follows that
\begin{equation}
\label{eq:est}
|\cli(S_x \cup C_x)| - |\cli(C_x)| ~\leq~ |C_x^{(d-1)}| \cdot |(S_x \cup C_x)^{(0)}|\,.
\end{equation}
 Since $C_x$ is a $d$-cycle, every $(d-1)$-face in it is adjacent to
two or more $d$-faces of $C_x$, while every $d$-face is adjacent to $(d+1)$ $d$-faces.
Thus, \;$|C_x^{(d-1)}| \leq {{d+1} \over 2} \cdot|C_x|$.

Consider now $|(S_x \cup C_x)^{(0)}|$. Since $(S_x \cup C_x)^{(0)}$ is constructed by starting with a $d$-cycle,
and repeatedly attaching to it $d$-ears, to borrow the terminology from the graph theory, any $d$-simplex in it is contained in a $d$-cycle. Therefore, by the basic property of $d$-cycles, every vertex in it is adjacent to at least $d+1$ $d$-simplices. Since every $d$-simplex is adjacent to $d+1$ vertices, one concludes that~ $|(S_x \cup C_x)^{(0)}| \leq |S_x \cup C_x)|$. Furthermore, since $S_x \cup C_x$ is the union of the cycles $C_y$, where $y$
appears on the path from the root of $\T_K$ to $x$, by Claim~\ref{cl:T_M},~
$|S_x \cup C_x)| \,\leq\, k+ (k-1) + \ldots + 1 \,=\, k(k+1)/2$.

To sum up,
\begin{equation}
\label{eq:est1}
|\cli(S_x \cup C_x)| - |\cli(S_x)| ~~\leq~~ {{d+1} \over 2} \cdot|C_x|\cdot |(S_x \cup C_x)|
~~\leq~~{{d+1} \over 4} \cdot (k+1)k \cdot |C_x|\,.
\end{equation}
Combining  (\ref{eq:size2}), (\ref{eq:est1}),
and using Claim~\ref{cl:T_M}, one arrives at
\[
f_d(K) ~~\leq~~ \sum_{x\in \T_M} |\cli(S_x \cup C_x)| - |\cli(S_x)| ~~\leq~~
\sum_{x\in \T_M} {{d+1} \over {4}} \cdot k(k+1)  \cdot |C_x|
\]
\[
~~\leq~~ {{d+1} \over {2}} \cdot k(k+1) \cdot \sum_{x\in \T_M} (|C_x|-1) ~~\leq~~  {{d+1} \over {2}} \cdot k(k+1) \cdot\rank_d(K)\,.
\]
\end{proof}
Let now $K$ be a pure simplicial $d$-complex with $c_d(K) > 0$. Since
$f_{d-1} \geq \rank_d(K)$, Theorem \ref{t:d=2} implies the following lower bound on $c_d(K)$ in terms of its density:
\begin{theorem}
\label{t:density}
For $K$ as above,
\[
c_d(K) ~\geq~  \sqrt{{2 \over {d+1}} \cdot { {f_d(K)} \over
    {f_{d-1}(K)}}} \;-\;1~.
\]
\end{theorem}
$\mbox{}$
\\ \\
{\bf \large Open Problems} The most intriguing open problem raised by this paper is the tightness of
the above lower bounds. Currently the dependence of $c(K)$ in
$\gamma(K)$ can be anything
between the lower bound in Theorem \ref{t:d=2} and the upper bound
achieved by the complete complex $K^{(d)}_n$ in which $c(K) = \theta(\gamma(K)^d)$.

 In addition, our lower bounds yield nothing when $\gamma_d(K) \leq d^3$;
clearly, if $c_d(K) > 1$, then $c_d(K) \geq d+1$. What happens in this range?
$\mbox{}$
\\ \\ \\
{\bf \large Acknowledgements} We are very grateful to Jan van den Heuvel from the London School of Economics for
an insightful discussion which led to the formulation of the main question addressed in this paper.
%
%
%
%
%

\end{document}